\documentclass[12pt]{amsart}

 \newtheorem{thm}{Theorem}[section]
 \newtheorem{cor}[thm]{Corollary}
 \newtheorem{lem}[thm]{Lemma}
 \newtheorem{prop}[thm]{Proposition}
 \theoremstyle{definition}
 \newtheorem{defn}[thm]{Definition}
 \theoremstyle{remark}
 \newtheorem{rem}[thm]{Remark}
 \theoremstyle{definition}
 
\newtheorem{nota}[thm]{Notation}

 \newcommand{\reviewtimetoday}[2]{
} 
 \reviewtimetoday{\today}{Draft Version}

\usepackage{xypic} \usepackage{amsthm}

\begin{document}

\title[Higher secant varieties of $\mathbb P^n \times\mathbb P^ 1$]
{Higher secant varieties of $\mathbb P^n \times\mathbb P^ 1$ embedded 
in bi-degree $(a,b)$}

\author[E. Ballico]{Edoardo Ballico}
\address[E. Ballico]{Dept. of Mathematics, 
 University of Trento, 
38123 Povo (TN), Italy}
\email{ballico@science.unitn.it}

\author[A. Bernardi]{Alessandra Bernardi}
\address[A. Bernardi]{INRIA(Institut National de Recherche en Informatique et en Automatique), Project Galaad, 2004 route des Lucioles, BP 93, 06902 Sophia Antipolis, France.}
\email{alessandra.bernardi@iria.fr}

\author[M.V.Catalisano]{Maria Virginia Catalisano}
\address[M.V.Catalisano]{DIPTEM - Dipartimento di Ingegneria della Produzione, Termoenergetica e Modelli
Matematici, Universit\`{a} di Genova, Piazzale Kennedy, pad. D 16129 Genoa, Italy.}
\email{catalisano@diptem.unige.it}


\maketitle


\begin{abstract} In this paper we compute the dimension of all the higher secant varieties to the Segre-Veronese embedding of $\mathbb{P}^n\times \mathbb{P}^1$ via the section of the sheaf $\mathcal{O}(a,b)$ for any $n,a,b\in \mathbb{Z}^+$. We relate this result to the Grassmann Defectivity of Veronese varieties and we classify all the  Grassmann $(1,s-1)$-defective Veronese varieties.
\end{abstract}

\section*{Introduction}

Let $X\subset \mathbb{P}^N$ be a projective non degenerate variety of dimension $n$.   We will always work with projective spaces defined over  an
algebraically closed field $K$ of characteristic $0$. The Zariski closure of the union of all the $(s-1)$-dimensional projective spaces that are $s$ secant to $X$ is called the {\it $s^{th}$ secant variety of $X$} that we will denote with $\sigma_s(X)$ (see Definition \ref{secant}). For these  varieties it is immediate to find an expected dimension that is $exp\dim \sigma_s(X) =\min \{ sn+s-1 , N \}$. The so called  ``Terracini's Lemma" (we recall it here in  Lemma \ref{TerrLemma}) shows that actually $\dim \sigma_s(X) \leq exp\dim \sigma_s(X)$. The varieties $X\subset \mathbb{P}^N$ for which $\dim \sigma_s(X) < exp\dim \sigma_s(X) $ are said to be $(s-1)$-{\it defective}. For a fixed variety $X\subset \mathbb{P}^N$, a complete classification of all the defective cases is known only if $X$ is a Veronese variety (see \cite{AH95}). 

The main result of this paper is the complete  classification of all the defective secant varieties of the Segre-Veronese varieties $X_{n,1,a,b}$ that can be obtained by embedding $\mathbb{P}^n\times \mathbb{P}^1$ with the section of the sheaf $\mathcal{O}(a,b)$ for any couple of positive integer $(a,b)$ (see Theorem \ref{mainteor}).  

Nowadays this kind of problems turns out to be of certain relevance also in the applications like Blind Identification in Signal Processing and Independent Component Analysis (see eg. \cite{ComoR06:SP} and \cite{Yere10:aalborg}) where it is often required to determine the dimensions of the varieties parameterizing partially symmetric tensor of certain rank. The Segre-Veronese embedding of $\mathbb{P}^n\times \mathbb{P}^1$  with the section of the sheaf $\mathcal{O}(a,b)$ parameterizes in fact rank 1 partially symmetric tensors of order $a+b$ contained in $\mathbb{P}(S^aV_1\otimes S^bV_2)$ where $V_1,V_2$ are two vector spaces of dimensions $n+1$ and $2$ respectively and with $S^dV$ we mean the space of symmetric tensors over a vector space $V$. The generic element of the $s^{th}$ secant variety of Segre-Veronese varieties is indeed a partially symmetric tensor of rank $s$ (see eg. \cite{MR2391665} for a detailed description).

The study of the dimensions of the secant varieties of Segre-Veronese varieties of two factors begun with F. London (see \cite{London} and  \cite{DF03}, 
\cite{CaCh} for a more recent approach)
who studied the case of  $\mathbb{P}^ 1 \times \mathbb{P}^ 2$
embedded in bi-degree $(1,3)$.
A  first generalization for $\mathbb{P}^ 1 \times \mathbb{P}^ 2$
embedded in bi-degree $(1,d)$ is given in \cite{DF03}. The
case for $\mathbb{P}^ 1 \times \mathbb{P}^ 2$ embedded in any bi-degree $(d_{1},
d_{2})$ is done by K. Baur and J. Draisma in \cite{BD}. L. Chiantini e C. Ciliberto in \cite{ChCi} handle with the case $\mathbb{P}^ 1
\times \mathbb{P}^ n$ embedded in bi-degree $ (d,1)$.
In the paper \cite{CGG4}  one can find  the cases $\mathbb{P}^ m \times \mathbb{P}^ n$
with bi-degree $ (n+1,1)$, $\mathbb{P}^ 1 \times \mathbb{P}^ 1$ with bi-degree $
(d_{1}, d_{2})$ and $\mathbb{P}^ 2 \times \mathbb{P}^ 2 $ with bi-degree $(2,2)$.
In \cite{Abrescia} S. Abrescia studies  the cases $\mathbb{P}^ n \times \mathbb{P}^ 1 $ in bi-degree
$(2,2d+1)$, $\mathbb{P}^ n \times \mathbb{P}^ 1$ in bi-degree $ (2,2d)$, and $\mathbb{P}^
n \times \mathbb{P}^ 1 $ in bi-degree $ (3,d)$. 
In \cite{BCC} the authors study the cases of $\mathbb{P}^n \times \mathbb{P}^m$ embedded in bi-degree $(1,d)$.

A recent
result on  $\mathbb{P}^ n \times \mathbb{P}^ m$ in bi-degree $ (1,2)$ is in
\cite{ABo}, where H. Abo and M. C. Brambilla prove the existence of two functions
$\underline{s}(n,m)$ and $\overline{s}(n,m)$ such that
$\sigma_s(X_{(n,m,1,2)})$ has the expected dimension for
$s\leq \underline{s}(n,m)$ and for $s\geq \overline{s}(n,m)$. In
the same paper it is also shown that $X_{(1,m,1,2)}$ is never
defective and all the defective cases for $X_{(2,m,1,2)}$ are
described. In  \cite{A10} H.  Abo is interested in  the cases  $\mathbb{P}^ n \times \mathbb{P}^ m$, when $m = n$ or $m = n+1$.

In \cite[Theorem 1.2] {AB09} H. Abo and M. C.  Brambilla  study the cases of  $\mathbb{P}^n \times \mathbb{P}^1$ embedded in bi-degree $(a,b)$, for $b \geq 3$. Our Theorem \ref{mainteor} extends their result, eliminating the restricting hypothesis $b \geq 3$.
Our proof is not based on their result, moreover, since our theorem cover the case $b=1$, we may relate our result   with the notion of Grassmann defectivity (see Section \ref{Gd} and Definition \ref{GrassDef}). That allows to translate that result in terms of the number of homogeneous polynomials of certain degree $a$ that can be written Êas linear combination of the same $a$-th powers of linear forms (see Corollary \ref{cor1}).

Regarding the papers that treat the Êknowledge of Grassmann defectivity we quote the following ones: \cite{ChCi}  proves that the curves are never $(k,s-1)$-defective, \cite{ChCi05}  completely classify the case of $(1,k)$-defective surfaces and Ê\cite{Cop} studies the case of $(2,3)$-defective threefolds. Our Corollary \ref{cor1} fits in this framework showing that the Veronese varieties $\nu_a(\mathbb{P}^n)$ are $(1,s-1)$-defective if and only if $a=3$ and $s=5$.

\section{Preliminaries and Notation}\label{prelim}

Let us
recall the notion of higher secant varieties and some
 classical results which we will often use.
 For definitions and proofs we refer the reader to \cite{CGG4}.

\begin{defn}\label{secant} {\rm Let $X\subset \mathbb P^N$ be a projective variety.
We define the
 $s^{th}$ {\it higher secant variety}  of $X$, denoted by $\sigma_{s}(X)$, as the Zariski closure of the union of all linear spaces spanned by  $s$ points of $X$, i.e.:
$$\sigma_{s}(X):= \overline{ \bigcup_{P_{1}, \ldots , P_{s}\in  X} \langle P_{1}, \ldots , P_{s} \rangle}\subset \mathbb P^ N.$$
The expected dimension of $\sigma_{s}(X)$  is
$$exp\dim \sigma_s(X) = \min\{N, s(\dim X+1)-1\},$$
and when $\sigma_s(X)$ does not have the expected dimension, $X$ is
said to be $(s-1)$-{\it defective}, and the positive integer
$$
\delta _{s-1}(X) = {\rm min} \{N, s(\dim X+1)-1\}-\dim \sigma_s(X)
$$
is called the  $(s-1)${\it-defect} of $X$. }\end{defn} 

\begin{rem}

Observe that classical authors used the notation  $Sec_{s-1}(X)$ for 
 the closure of the union of all linear 
spaces spanned by $s$ independent points of $X$, that is, for the object that we denote with $\sigma_s(X)$.

\end{rem}
The basic
tool to compute the dimension of $\sigma_{s}(X)$ is Terracini's
Lemma (\cite{Terracini}):

\begin{lem}[{\bf Terracini's Lemma}]\label{TerrLemma}

 Let $X\subset \mathbb P^ N$  be a projective, irreducible
variety, and let $P_1,\ldots ,P_s \in X$  be s generic
points. Then the projectivized tangent space to $\sigma_{s}(X)$ at a
generic point $Q \in \langle P_1,\ldots ,P_s \rangle$  is the linear span in
$\mathbb P^ N$ of the tangent spaces $T_{X, P_i}$ to $X$ at $P_i$,
$i=1,\ldots ,s$, hence
$$ \dim \sigma_{s}(X) = \dim  \langle T_{X,P_1},\ldots ,T_{X,P_s}\rangle.$$
\end{lem}

\begin{nota}\label{SV} We now fix the notation that we will adopt in this paper.

For all non-negative integers $n,m,a,b$,  let  $N= {n+a \choose n} {m+b \choose m}-1,$ and let $X_{(n,m,a,b)} \subset \mathbb P^N$  be   the embedding of $\mathbb P^n
\times \mathbb P^m$ into  $\mathbb P^N$ via the sections of the sheaf ${\mathcal O}(a, b)$.
This variety is classically known as a {\it two-factors Segre-Veronese variety}.

 Now define the integers $q_{(n,m,a,b)}$,  $r_{(n,m,a,b)}$, and  $q^*_{(n,m,a,b)}$:
  $$
 q_{(n,m,a,b)}=\left  \lfloor \frac{      {n+a\choose n}  {m+b\choose m}          }   {       (n+m+1)     }     \right  \rfloor
 ;
 $$
$$
r_{(n,m,a,b)}= {n+a\choose n}  {m+b\choose m}- (n+m+1)  q_{(n,m,a,b)};
 $$
 $$q^*_{(n,m,a,b)} =\left \lceil   \frac{ {n+a\choose n}  {m+b\choose m}} { (n+m+1)  }\right  \rceil = \left \{ 
 \begin{matrix}   q_{(n,m,a,b)} &\  {\rm  for } \  \ r_{(n,m,a,b)} =0  \\
\\
 q_{(n,m,a,b)} +1 & {\rm  for }  \  \ r_{(n,m,a,b)}  >0  \\
\end{matrix}
  \right. .
  $$
 Moreover, let us introduce the following two integers:
$$
e_{(n,m,a,b)}= \max \{s \in \mathbb N  \  \ { \rm such \  that } \  \dim \sigma_s(X_{(n,m,a,b)} )= s(n+m+1)-1\},
$$
$$
e^*_{(n,m,a,b)}= \min \{s \in \mathbb N  \  \ { \rm such \  that } \  \dim \sigma_s(X_{(n,m,a,b)} )= N\}.
$$
\begin{rem}

It is  easy to see that
$$e_{(n,m,a,b)} \leq q_{(n,m,a,b)} \leq q^*_{(n,m,a,b)}  \leq  e^*_{(n,m,a,b)}.$$
For $e_{(n,m,a,b)} = q_{(n,m,a,b)}$ and $ q^*_{(n,m,a,b)}=  e^*_{(n,m,a,b)}$,
 for any $s$ we have:
$$\dim \sigma_s(X_{(n,m,a,b)} )=
\min\{N, s(n+m+1)-1\},$$
 that is, the expected dimension. 
 
 Hence, 
 if $e_{(n,m,a,b)} = q_{(n,m,a,b)}$ and $ q^*_{(n,m,a,b)}=  e^*_{(n,m,a,b)}$, then
$\sigma_s(X_{(n,m,a,b)} ) $ is never defective.

\end{rem}

\end{nota}
\medskip
\medskip

A consequence of Terracini's Lemma is the following Corollary (see
\cite[Section 1]{CGG4} or  \cite [Section 2] {ABo} for a proof of
it).

\begin {cor}\label{corTer}
Let  $R = K [x_0, \dots, x_n,y_0, \dots, y_m]$ be the multigraded
coordinate ring of $ \mathbb P^n\times \mathbb P^m$, 
let $Z \subset \mathbb P^n \times \mathbb P^m$ be a set of $s$
generic double points, 
let $I_Z \subset R $ be the multihomogeneous ideal of $Z$, and let  $ H
(Z,(a,b)) $ be the multigraded Hilbert function of $Z$. Then
$$
\dim   \sigma_s \left ( X_{(n,m,a,b)}   \right ) = N - \dim (I_Z)_{(a,b)} =
 H (Z,(a,b))  -1 .
$$ 
\end{cor}

Now we recall  the fundamental  tool which allows us to convert
certain questions about ideals of varieties in multiprojective
space to questions about ideals in standard polynomial rings (for
a more general statement see \cite[Theorem 1.1]  {CGG4}) .

\begin{thm}\label{metodoaffineproiettivo} Notation as in \ref{SV},
let $X_{(n,m,a,b)}\subset \mathbb P^N$   be a two factor Segre-Veronese variety, and let $Z \subset
\mathbb P^n \times \mathbb P^m$  be a set of $s$
generic double points. Let
$H_{1}, H_{2}\subset \mathbb P^ {n+m}$ be  generic projective linear
spaces of dimensions $n-1$ and $m-1$, respectively, and let
$P_{1}, \dots ,P_{s} \in \mathbb P^{n+m}$  be
generic points. Denote by
$$bH_1+aH_2+2P_{1}+ \cdots + 2P_s \subset \mathbb P^{n+m}$$
 the scheme defined
by the ideal sheaf
${\mathcal I}^{b}_{H_1}\cap {\mathcal I}_{H_2}^a \cap {\mathcal I}^{2}_{P_1}\cap \dots \cap
{\mathcal I}^{2}_{P_s}\ \subset {\mathcal O}_{P^{n+m }  }$.
Then
$$
 \dim (I_Z)_{(a,b)} = \dim (I_{bH_1+aH_2+2P_1+ \cdots + 2P_s})_{a+b}
$$
hence
$$
\dim  \sigma_s \left ( X_{(n,m,a,b)}  \right )  = N - \dim (I_{bH_1+aH_2+2P_1+ \cdots + 2P_s})_{a+b}  .
$$
\end{thm}

\begin{rem}\label{rem1} Since 
$\dim  (I_{bH_1+aH_2})_{a+b} ={n+a\choose n}  {m+b\choose m}$, then the expected dimension of $(I_{bH_1+aH_2+2P_1+ \cdots + 2P_s})_{a+b} $ is 
$$\max \left \{{n+a\choose n}  {m+b\choose m}- s(n+m+1); 0 \right  \}.$$
 It follows that 
$ \sigma_s \left ( X_{(n,m,a,b)}   \right ) $ is not defective if and only if the part of degree $a+b$ of the ideal
$  (I_{bH_1+aH_2+2P_1+ \cdots + 2P_s})$ has the expected dimension.

\end{rem}

Since we will make use of  Horace's differential lemma several times,
we recall it here  (for notation and proof we refer to \cite{AH00},
Section 2).

 \begin{lem} [{\bf Horace's differential lemma}] \label{Horacediff} 
Let $H \subseteq \mathbb P^t$ be a hyperplane, and let $P_1,...,P_r$ be generic points in $\mathbb P^ t$. 
\par
Let $Z+ 2P_1 + ... + 2P_r\subset \mathbb P^ t$ be a  scheme . Let
$R_1,...,R_r \in H$ be generic  points, and set
$$W= Res _H Z+ 2R_1|_H + ... + 2R_r|_H
\subset \mathbb P^t , $$ 
 $$T=Tr _H  Z+ R_1 + ... + R_r \subset \mathbb P^ {t-1} \simeq H.$$
If the following two conditions are satisfied:
\par
\noindent
${\bf Degue:} \hskip 1cm \dim (I_{W})_{d-1}=0,$
 \par
\noindent${\bf Dime:} \hskip 1.2cm \dim (I_{T})_{d}= 0,$\par
\noindent
 then 
 $$\dim (I_{Z+ 2P_1 + ... + 2P_r})_d=0 .$$
\end{lem}

Now we recall the following useful lemma. It gives a criterion for adding to a scheme $X \subseteq \mathbb P^ N$ a set of reduced points lying on a linear space ${H} \subset \mathbb P ^N$ and imposing independent
conditions to forms of a given degree in the ideal of $X$ (see  \cite[Theorem 1.1]  {CGG4}).

 \begin{lem} \label{AggiungerePuntiSuSpazioLineare} Let $d \in \mathbb N$. Let $X  \subset \mathbb P^N$ be  a scheme, and
let $P_1,\dots,P_s$ be $s$ generic distinct points lying on a  linear space
$H \subset \mathbb P^N$.

If $\dim (I_{X })_d =s$ and $\dim (I_{X +H})_d =0$, then
$
\dim (I_{X+P_1+\cdots+P_s})_d = 0.$

\end{lem}

\medskip Finally we recall a list of results, which we  will use to prove our main theorem
 (see Theorem \ref{mainteor}).
  We start with a result due to  M. V. Catalisano, A. V. Germita and A. Gimigliano  about  the higher secant varieties of $\mathbb P^1\times \mathbb P^1$:

 \begin{thm}   \cite  [Corollary 2.3] {CGG4}  \label{teoremaCatGerGim}   Let  $a,b,d,s$ be positive integers, then
 $ \sigma_s(X_{(1,1,a,b)} ) $
  has the expected dimension
  except for 
$$
(a,b)=(2,2d) \ ({\it or} \  (a,b)=(2d,2)) \ 
\ \ {\it and} \ \ s=2d+1 \ .
$$
In this cases $$\dim \sigma_s(X_{(1,1,a,b)} ) =3s-2$$
and the $s${\it-defect} of $X_{(1,1,a,b)} $ is 1.
 
 \end{thm}

  K. Baur and J. Draisma, by a tropical approach,
studied the case $\mathbb P^2\times \mathbb P^1$:
 
  \begin{thm}  \cite  [Theorem 1.3] {BD} \label{teoremaBaurDraisma}   Let  $a,b,d,s$ be positive integers, then
 $ \sigma_s(X_{(2,1,a,b)} ) $
  has the expected dimension
  except for 
  $$
(a,b)=(3,1) 
\ \ {\it and} \ \ s=5 \ ;
$$
$$
(a,b)=(2,2d) \
\ \ {\it and} \ \ s=3d+1 \ \   {\it or }  \ \  s=3d+2\ .
$$
In this cases  the $s${\it-defect} of $X_{(2,1,a,b)} $ is 1.
 \end{thm}

 L. Chiantini e C. Ciliberto in \cite{ChCi} studied the Grassmannians of secant varieties of irreducible curves  and they show that these varieties always have the expected dimension.  As an immediate consequence of  their \cite[Theorem 3.8]{ChCi} and of a classical result about  Grassmann secant varieties
 (see Proposition \ref {CarlaClaudio}) we immediately get  the regularity of the higher secant varieties of $\mathbb P^n\times \mathbb P^1$, embedded with divisors of type $(1,b)$ (see also \cite[Theorem 3.5]{AB09}):
 
 \begin{thm}\label{teoremaChiantiniCil} 
 Let  $n,b,s$ be positive integers, then
 $ \sigma_s(X_{(n,1,1,b)} ) $
  has always the expected dimension.

\end{thm}

S. Abrescia studied the cases  $\mathbb P^n\times \mathbb P^1$, embedded with divisors of type  $(2,b)$ and $(3,b)$:

 \begin{thm}   \cite[Proposition 3.1 and Theorem 3.4] {Abrescia}  \label{teoremaAbrescia1} 
 Let  $n,b,d,s$ be positive integers, then
 $ \sigma_s(X_{(n,1,2,b)} ) $
  has the expected dimension
  except for 
$$
b=2d \
\ \ {\it and} \ \ d(n+1)+1 \leq s \leq (d+1)(n+1)-1 .
$$
In this cases $\sigma_s(X_{(n,1,2,b)} )$ is defective, and 
its defectiveness is $\delta_s={d(n+1)\choose 2}+{s+1 \choose 2}-sd(n+1)$.

\end{thm}

 \begin{thm}  \cite[Theorem 4.2] {Abrescia}  \label{teoremaAbrescia2} 
 Let  $n,b,s$ be positive integers, then
 $ \sigma_s(X_{(n,1,3,b)} ) $
  has the expected dimension
  except for 
$$
n=2 , \ \  b=1 
\ \ {\it and} \ \ s=5 .
$$
In this case $\sigma_s(X_{(2,1,3,1)} )$ is defective, and 
its defectiveness is 1.

\end{thm}


  \section{Technical lemmata}

This section contains several technical lemma  which will be used in the proof of the main theorem in the next section. From now on we consider only the case $m=1$,
that means that we focus our attention on the Segre-Veronese variety $X_{(n,1,a,b)}$.


 \begin{lem}\label{lemma1} Notation as in \ref{SV}, let $n,a,b$ be positive integers,   $n \geq 2$, $a \geq 2$,  $s \leq q_{(n,1,a,b)}$. If

{\rm (1) } $q_{(n-1,1,a,b)} = e_{(n-1,1,a,b)}$;

{\rm (2) }  $q_{(n,1,a,b)} \geq q_{(n-1,1,a,b)}+r_{(n-1,1,a,b)}$;

{\rm (3) }  $e_{(n,1,a-1,b)} \geq q_{(n,1,a,b)}-q_{(n-1,1,a,b)}$;

{\rm (4) }  $e^*_{(n,1,a-2,b)} \leq q_{(n,1,a,b)}-q_{(n-1,1,a,b)}-r_{(n-1,1,a,b)}$;

then $ \sigma_s(X_{(n,1,a,b)} ) $ has the expected dimension, that is,
$$\dim\sigma_s(X_{(n,1,a,b)} ) =s(n+2)-1.
$$

\end{lem}

\begin{proof}
It is enough to prove that  for $s = q_{(n,1,a,b)}$, the variety $ \sigma_s(X_{(n,1,a,b)} ) $ has the expected dimension.

Let $H_{1}, H_{2}$ be as in Theorem \ref{metodoaffineproiettivo}, so
$H_{1}, H_{2}\subset \mathbb P^ {n+1}$ are  generic projective linear
spaces. More precisly,   $H_{1} \simeq \mathbb P^{n-1}$ and $H_2$ is a point. Let
$P_{1}, \dots ,P_{q_{(n,1,a,b)}}, $ $P'_{1}, \dots ,P'_{r_{(n,1,a,b)}}\in \mathbb P^{n+1}$  be
generic points and set

$$\mathbb{X} = bH_1+aH_2+2P_{1}+ \cdots + 2P_{q_{(n,1,a,b)}}+ P'_{1}+\dots +P'_{r_{(n,1,a,b)}} \subset \mathbb P^{n+1}.$$

Let  $H\subset \mathbb P^{n+1}$  be  a generic hyperplane through the point $H_2$. Now specialize the points 
$P_{1}, \cdots ,P_{q_{(n-1,1,a,b)}}$ on  $H$, and denote by 
$\widetilde P_{1}, \cdots ,\widetilde P_{q_{(n-1,1,a,b)}}$ the specialized points and by  $\widetilde {\mathbb X} $ the specialized scheme. 

Now denote by Z the scheme
 $$Z= bH_1+aH_2+2 \widetilde P_{1}+ \cdots + 2 \widetilde P_{q_{(n-1,1,a,b)}}
$$
$$
+ 2P_{q_{(n-1,1,a,b)}+1}  +  \dots  + 2P_{q_{(n,1,a,b)} - r_{(n-1,1,a,b)}} 
+ P'_{1}+\dots +P'_{r_{(n,1,a,b)}}.
$$
(Observe that this is possible, in fact, by the inequality { (2)},  we have $q_{(n,1,a,b)}-r_{(n-1,1,a,b) \geq q_{(n-1,1,a,b)}}$).
By this notation we have
$$
\widetilde {\mathbb X} =Z+  2P_{q_{(n,1,a,b)} - r_{(n-1,1,a,b)}  +1} + \dots +  2P_{q_{(n,1,a,b)} } ,
$$
that is, $\widetilde {\mathbb X} $ is union of $Z$ and $r_{(n-1,1,a,b)} $ generic double points.

We will  apply the Horace's differential lemma  (see Lemma \ref{Horacediff})
 to the scheme $ \widetilde {\mathbb X}$,
with $t=n+1$, $d=a+b$, $r= r_{(n-1,1,a,b)}$.

Let $R_1,...,R_{r_{(n-1,1,a,b)} } $ 
be generic  points in $H$,  and let
$$W= Res _H Z+ 2R_1|_H + ... +2R_{r_{(n-1,1,a,b)} } |_H
\subset \mathbb P^{n+1} , $$ 
 $$T=Tr _H  Z+ R_1 + ... + R_{r_{(n-1,1,a,b)} }   \subset \mathbb P^ {n} \simeq H.$$
If we will prove that \par
\noindent
${\bf Degue:} \hskip 1cm \dim (I_{W})_{a+b-1}=0,$
 \par
\noindent${\bf Dime:} \hskip 1.2cm \dim (I_{T})_{a+b}= 0,$\par
\noindent
then  by Lemma \ref{Horacediff} we will have
$
 \dim (I_{\widetilde {\mathbb X} })_{a+b}=0. $

It is easy to see that
$$T= bH_1'+aH_2+2 \widetilde P_{1}|_H+ \cdots + 2 \widetilde P_{q_{(n-1,1,a,b)}}|_H
+ R_1 + ... + R_{r_{(n-1,1,a,b)} } ,
$$
where $H_1'=H_1\cap H\simeq \mathbb P^{n-2}$.
Thus by (1), by Theorem \ref {metodoaffineproiettivo} and by Remark  \ref {rem1}  we immediately get 
$$\dim (I_{T})_{a+b}= 0.$$
Now we calculate  the dimension of $ I_{W}$ in degree $a+b-1$. We have
$$\begin{array}{rl}
W= &  Res _H(bH_1+aH_2+2 \widetilde P_{1}+ \cdots + 2 \widetilde P_{q_{(n-1,1,a,b)}}\\
&+ 2P_{q_{(n-1,1,a,b)}+1}  +  \dots  + 2P_{q_{(n,1,a,b)} - r_{(n-1,1,a,b)}} 
+ P'_{1}+\dots +P'_{r_{(n,1,a,b)}} )\\
&+ 2R_1|_H + ... + 2R_{r_{(n-1,1,a,b)} } |_H \\
=&  bH_1+(a-1)H_2+ \widetilde P_{1}+ \cdots +  \widetilde P_{q_{(n-1,1,a,b)}}\\
&+ 2P_{q_{(n-1,1,a,b)}+1}  +  \dots  + 2P_{q_{(n,1,a,b)} - r_{(n-1,1,a,b)}} \\
&+ P'_{1}+\dots +P'_{r_{(n,1,a,b)}} +2R_1|_H + \cdots + 2R_{r_{(n-1,1,a,b)} } |_H.
\end{array}
$$

In order to compute the dimension of $\dim( I_{W})_{ a+b-1 }$, we first  consider the scheme:
$$\begin{array}{rl}
W' =  & bH_1+(a-1)H_2
+ 2P_{q_{(n-1,1,a,b)}+1}  +  \dots  + 2P_{q_{(n,1,a,b)} - r_{(n-1,1,a,b)}} \\
&
+2R_1 + ... + 2R_{r_{(n-1,1,a,b)} }.
\end{array}
$$
The points $H_2, R_1, \dots, R_{r_{(n-1,1,a,b)} }$ lie on $H\simeq \mathbb P^{n}$, 
but 
$$\sharp \{H_2, R_1, \dots, R_{r_{(n-1,1,a,b)} }\}\leq n+1,$$ hence  we may say that $W' $ is union of $bH_1+(a-1)H_2$ and 
$q_{(n,1,a,b)} -q_{(n-1,1,a,b)} $ generic double points. So by (3) we know that the dimension of $I_{W'}$ in degree $a+b-1$ is as expected. It follows that  the dimension in degree $a+b-1$ of its  subscheme 
$$\begin{array}{rl}
W''= & bH_1+(a-1)H_2
+ 2P_{q_{(n-1,1,a,b)}+1}  +  \dots  + 2P_{q_{(n,1,a,b)} - r_{(n-1,1,a,b)}} \\
&+2R_1|_H + ... + 2R_{r_{(n-1,1,a,b)} }|_H.
\end{array}
$$
 is also as expcted, that is,
 $$\dim (I_{W''})_{a+b-1}
 $$
 $$=
 {n+a-1\choose n}  (b+1)-   (n+2  )  (q_{(n,1,a,b)} -q_{(n-1,1,a,b)} )+r_{(n-1,1,a,b)} .
 $$
From here, since 
$$W=  W'' + \widetilde P_{1}+ \cdots +  \widetilde P_{q_{(n-1,1,a,b)}}+  P'_{1}+\dots +P'_{r_{(n,1,a,b)}}$$
and since  the points  $P'_{1},\dots ,P'_{r_{(n,1,a,b)}} $ are generic, we easily get that 
$$
\dim (I_{W- (\widetilde P_{1}+ \cdots +  \widetilde P_{q_{(n-1,1,a,b)}} )}  )_{a+b-1} 
$$
$$
={n+a-1\choose n}  (b+1)-   (n+2  )  (q_{(n,1,a,b)} -q_{(n-1,1,a,b)} )+r_{(n-1,1,a,b)} -r_{(n,1,a,b)}
$$
$$=q_{(n-1,1,a,b)}.
$$
Now, in order to prove that $\dim (I_W)_{a+b-1}=0$, or, equivalently, that the $q_{(n-1,1,a,b)}$ points $\widetilde P_{i}$ lying on $H$, gives independent conditions to the forms of degree $a+b-1$ of $I_ {W- (  \widetilde P_{1}+ \cdots +  \widetilde P_{q_{(n-1,1,a,b)}}  ) }$, we apply 
Lemma \ref{AggiungerePuntiSuSpazioLineare} with 
$$X= {W-  (  \widetilde P_{1}+ \cdots +  \widetilde P_{q_{(n-1,1,a,b)}}  )} . $$
Since obviously   $2R_i|_H \subset H,$   we have
$$\begin{array}{rl}
X+H = &  bH_1+(a-1)H_2
+ 2P_{q_{(n-1,1,a,b)}+1}  +  \dots  + 2P_{q_{(n,1,a,b)} - r_{(n-1,1,a,b)}}  \\
&+ P'_{1}+\dots +P'_{r_{(n,1,a,b)}} +2R_1|_H + \cdots + 2R_{r_{(n-1,1,a,b)} } |_H +H \\
 =&  bH_1+(a-1)H_2+H +
 2P_{q_{(n-1,1,a,b)}+1}  +  \dots  \\ 
 &+ 2P_{q_{(n,1,a,b)} - r_{(n-1,1,a,b)}} + P'_{1}+\dots +P'_{r_{(n,1,a,b)}} .
\end{array}
$$

So
$$\dim (I_{X+H})_{a+b-1} = \dim (I_{Res_H (X+H) })_{a+b-2} ,
$$
where
$$Res_H (X+H)=  bH_1+(a-2)H_2
$$
$$
+ 2P_{q_{(n-1,1,a,b)}+1}  +  \dots  + 2P_{q_{(n,1,a,b)} - r_{(n-1,1,a,b)}} 
+ P'_{1}+\dots +P'_{r_{(n,1,a,b)}} .$$
By (4) we immediately get that
$$\dim (I_{Res_H (X+H) })_{a+b-2}=0.$$ 
So, by 
Lemma \ref{AggiungerePuntiSuSpazioLineare}, it follows that $\dim (I_W)_{a+b-1}=0$.

Thus, by Lemma \ref{Horacediff}  we get $ \dim (I_{\widetilde {\mathbb X} })_{a+b}=0$, and from here, by  the semicontinuity of the Hilbert function we have  $\dim (I_ {\mathbb X} )_{a+b}=0$, that is, the expected dimension. 

Since the $P'_i$  are  generic points, it immediately follows that also 
$$\mathbb{X} -(P'_{1}+\dots +P'_{r_{(n,1,a,b)}}) = bH_1+aH_2+2P_{1}+ \cdots + 2P_{q_{(n,1,a,b)}} $$
 has the expected dimension. So by Remark \ref{rem1} we are done.

\end{proof}


 \begin{lem}\label{lemma2} Notation as in \ref{SV}, let $n,a,b$ positive integers,   $n \geq 2$, $a \geq 2$,  $s \geq q^*_{(n,1,a,b)}$. If

{\rm (1)}  $q_{(n-1,1,a,b)} = e_{(n-1,1,a,b)}$;

{\rm (2*)} $q^*_{(n,1,a,b)} \geq q_{(n-1,1,a,b)}+r_{(n-1,1,a,b)}$;

{\rm (3*)} $e_{(n,1,a-1,b)} \geq q^*_{(n,1,a,b)}-q_{(n-1,1,a,b)}$;

{\rm (4*)} $e^*_{(n,1,a-2,b)} \leq q^*_{(n,1,a,b)}-q_{(n-1,1,a,b)}-r_{(n-1,1,a,b)}$;

then $ \sigma_s(X_{(n,1,a,b)} ) $ has the expected dimension, that is,
$$\dim\sigma_s(X_{(n,1,a,b)} ) =N.
$$

\end{lem}

\begin{proof}
Observe that (1) in Lemma \ref{lemma1} is the same condition of (1) in this Lemma \ref{lemma2}. 

It is enough to prove that $\dim  \sigma_s(X_{(n,1,a,b)} ) =N $, for $s = q^*_{(n,1,a,b)}$.

For $r_{(n,1,a,b)}=0$, we have $q^*_{(n,1,a,b)}=q_{(n,1,a,b)}$, and the conclusion immediately follows from Lemma \ref{lemma1}.

Assume $r_{(n,1,a,b)}>0$, hence $q^*_{(n,1,a,b)}=q_{(n,1,a,b)}+1.$

Let $H_{1}, H_{2}$ be as in Theorem \ref{metodoaffineproiettivo}, so, as in the previous lemma,
$H_{1} \simeq \mathbb P^{n-1}$ and $H_2$ is a point. 

Let
$P_{1}, \dots ,P_{q^*_{(n,1,a,b)}}  \in \mathbb P^{n+1}$  be
generic points and set

$$\mathbb{X^*} = bH_1+aH_2+2P_{1}+ \cdots + 2P_{q^*_{(n,1,a,b)}}\subset \mathbb P^{n+1}.$$

Let  $H\subset \mathbb P^{n+1}$  be  a hyperplane passing through the point $H_2$ and, as in Lemma \ref{lemma1}, specialize the points 
$P_{1}, \cdots ,P_{q_{(n-1,1,a,b)}}$ on  $H$. Denote by 
$\widetilde P_{1}, \cdots ,\widetilde P_{q_{(n-1,1,a,b)}}$ the specialized points and by  $\widetilde {\mathbb X^*} $ the specialized scheme.

The proof is analogous to the one of Lemma \ref {lemma1}, hence we describe it briefly.

Since by (2*)  we have $q^*_{(n,1,a,b)}-r_{(n-1,1,a,b)} \geq q_{(n-1,1,a,b)}$,
we can consider the following scheme:
$$Z= bH_1+aH_2
$$
$$+2 \widetilde P_{1}+ \cdots + 2 \widetilde P_{q_{(n-1,1,a,b)}}
+ 2P_{q_{(n-1,1,a,b)}+1}  +  \dots  + 2P_{q^*_{(n,1,a,b)} - r_{(n-1,1,a,b)}}.
$$
We  apply the Horace's differential lemma  (see Lemma \ref{Horacediff}), with $t=n+1$ and  $d=a+b$,  to  $\widetilde {\mathbb X^*} $,  which is union of $Z$ and $r_{(n-1,1,a,b)} $ generic double points.

Let $R_1,...,R_{r_{(n-1,1,a,b)} } $ 
be generic  points in $H$,  and let
$$W= Res _H Z+ 2R_1|_H + ... +2R_{r_{(n-1,1,a,b)} } |_H
\subset \mathbb P^{n+1} , $$ 
 $$T=Tr _H  Z+ R_1 + ... + R_{r_{(n-1,1,a,b)} }   \subset \mathbb P^ {n} \simeq H.$$

Note that $T$ is the same scheme $T$ that appears in the proof of Lemma \ref {lemma1}, so
by (1), by Theorem \ref {metodoaffineproiettivo} and by Remark  \ref {rem1}  we get that
$$\dim (I_{T})_{a+b}= 0.$$
With regard to $W$ we have:
$$\begin{array}{rl}
W=  &  bH_1+(a-1)H_2+ \widetilde P_{1}+ \cdots +  \widetilde P_{q_{(n-1,1,a,b)}}
+ 2P_{q_{(n-1,1,a,b)}+1}  +  \dots  \\
&+ 2P_{q^*_{(n,1,a,b)} - r_{(n-1,1,a,b)}} 
+2R_1|_H + ... + 2R_{r_{(n-1,1,a,b)} } |_H.
\end{array}
$$
Let
$$
W' =  bH_1+(a-1)H_2
$$
$$
+ 2P_{q_{(n-1,1,a,b)}+1}  +  \dots  + 2P_{q^*_{(n,1,a,b)} - r_{(n-1,1,a,b)}} 
+2R_1 + ... + 2R_{r_{(n-1,1,a,b)} }.
$$
Now  $W' $ is union of $(bH_1+(a-1)H_2)$ and 
$q^*_{(n,1,a,b)} -q_{(n-1,1,a,b)} $ generic double points. So by (3*) the dimension of $(I_{W'})_{a+b-1}$ is as expected. 

It follows that  the dimension of the degree $a+b-1$ part of the following scheme:
$$W-(\widetilde P_{1}+ \cdots +  \widetilde P_{q_{(n-1,1,a,b)}})
= bH_1+(a-1)H_2$$
$$+ 2P_{q_{(n-1,1,a,b)}+1}  +  \dots  + 2P_{q^*_{(n,1,a,b)} - r_{(n-1,1,a,b)}} 
+2R_1|_H + ... + 2R_{r_{(n-1,1,a,b)} }|_H,$$
 is also as expcted. That is, (recall that $q^*_{(n,1,a,b)}= q_{(n,1,a,b)}+1$)
 $$\dim (I_{ W-(\widetilde P_{1}+ \cdots +  \widetilde P_{q_{(n-1,1,a,b)}}) })_{a+b-1}
 $$
 $$=
 {n+a-1\choose n}  (b+1)-   (n+2  )  (q^*_{(n,1,a,b)} -q_{(n-1,1,a,b)} )+r_{(n-1,1,a,b)} .
 $$
$$=r_{(n,1,a,b)} -(n+2)+  q_{(n-1,1,a,b)} \leq q_{(n-1,1,a,b)} -1
.$$
In order to prove that $\dim (I_W)_{a+b-1}=0$, we apply 
Lemma \ref{AggiungerePuntiSuSpazioLineare} with 
$$X= {W-  (  \widetilde P_{1}+ \cdots +  \widetilde P_{q_{(n-1,1,a,b)}}  )} . $$
We have
$$X+H =  bH_1+(a-1)H_2+H
$$
$$
+ 2P_{q_{(n-1,1,a,b)}+1}  +  \dots  + 2P_{q^*_{(n,1,a,b)} - r_{(n-1,1,a,b)}} .
$$

So
$$\dim (I_{X+H})_{a+b-1} = \dim (I_{Res_H (X+H) })_{a+b-2} ,
$$
where
$$Res_H (X+H)=  bH_1+(a-2)H_2
+ 2P_{q_{(n-1,1,a,b)}+1}  +  \dots  + 2P_{q^*_{(n,1,a,b)} - r_{(n-1,1,a,b)}} .$$
By (4*) we immediately get $\dim (I_{Res_H (X+H) })_{a+b-2}=0,$ so by 
Lemma \ref{AggiungerePuntiSuSpazioLineare} we have that $\dim (I_W)_{a+b-1}=0$.

Thus, by Lemma \ref{Horacediff}  we get $ \dim (I_{\widetilde {\mathbb X^*} })_{a+b}=0$, and from here, by  the semicontinuity of the Hilbert function we have  $\dim (I_ {\mathbb X^*} )_{a+b}=0$, that is, the expected dimension. Finally by Remark \ref{rem1} we get the conclusion.

\end{proof}

 \begin{lem}\label{lemma3} Notation as in \ref{SV}, let $n,a,b,s$ be positive integers,   $n \geq 2$, $a \geq 2$. Let {\rm (1)},  {\rm (3*)},  {\rm (4)} be as in the statements of Lemma \ref{lemma1} and Lemma \ref{lemma2}, that is,
 
(1)  $q_{(n-1,1,a,b)} = e_{(n-1,1,a,b)}$;

(3*) $e_{(n,1,a-1,b)} \geq q^*_{(n,1,a,b)}-q_{(n-1,1,a,b)}$;

(4) $e^*_{(n,1,a-2,b)} \leq q_{(n,1,a,b)}-q_{(n-1,1,a,b)}-r_{(n-1,1,a,b)}$;

then $ \sigma_s(X_{(n,1,a,b)} ) $ has the expected dimension, that is,
$$\dim\sigma_s(X_{(n,1,a,b)} ) =\min\{N, s(n+2)-1\}.
$$

\end{lem}

\begin{proof}
Let also (2), (3), (2*), (4*) be as in the statements of Lemmas \ref{lemma1} and  \ref{lemma2}.
Since
$(2) \Rightarrow (2^*) $, 
$(3^*)  \Rightarrow(3)$,
$(4) \Rightarrow (4^*) $, 
$(4) \Rightarrow (2) $, then the conclusion immediately follows from 
 Lemmas \ref{lemma1} and \ref{lemma2}.

\end{proof}

\section{Segre-Veronese embeddings of $\mathbb P^n \times \mathbb P^1$ }\label{results}

Now that we have introduced all the necessary tools that we need for the main theorem of this paper, we are ready to state and to prove it.
In Section \ref{prelim}, Theorems \ref{teoremaCatGerGim}, \ref{teoremaBaurDraisma}, \ref{teoremaChiantiniCil}. \ref{teoremaAbrescia1},  \ref{teoremaAbrescia2}  
 highlighted several defective higher secant varieties  of $\mathbb P^n\times \mathbb P^1$, embedded with divisors of type $(a,b)$. The next theorem shows that those are the 
only defective cases.
As already observed in the introduction, we remind that in  \cite{AB09} this result  is obtained for the case $b\geq 3$.

\begin{thm}\label{mainteor}
 Notation as in \ref{SV}, let $n,a,b,s,d$ be positive integers.

The variety $\sigma_{s} \left ( X_{(n,1;a,b)}\right ) \subset \mathbb P^N$, ($N  = {n+a \choose n}(b+1)-1$), 
 has the expected dimension for any $n,a,b,s$, except in the following cases:
 $$n=2,\ \ (a,b)=(3,1),\ \ \  s=5,$$
  and
 $$(a,b)=(2,2d) ,\ \ \   d(n+1)+1 \leq s \leq (d+1)(n+1)-1  .$$
 
\end{thm}

\begin{proof}
We will prove the theorem by induction on $n+a$. For $n=1$ and $n=2$ the conclusion follows from Theorem \ref{teoremaCatGerGim} and from Theorem \ref{teoremaBaurDraisma}, respectively.
If $a=1,2,3$ we get the conclusion by Theorems \ref{teoremaChiantiniCil},  \ref{teoremaAbrescia1}, \ref{teoremaAbrescia2}. 
So assume $n \geq 3$ and $a \geq 4$.

We want to apply Lemma \ref{lemma3}. We will check that $n,a,b,s$ verify the inequalities (1), (3*) and (4) of  Lemma \ref{lemma3}, that is:

(1)  $q_{(n-1,1,a,b)} = e_{(n-1,1,a,b)}$;

(3*) $e_{(n,1,a-1,b)} \geq q^*_{(n,1,a,b)}-q_{(n-1,1,a,b)}$;

(4) $e^*_{(n,1,a-2,b)} \leq q_{(n,1,a,b)}-q_{(n-1,1,a,b)}-r_{(n-1,1,a,b)}$.
 
The inequality (1) holds by induction. 
Moreover, by the induction hypothesis, we have
$$e_{(n,1,a-1,b)} = q_{(n,1,a-1,b)} .
$$
Hence in order to show that  (3*) holds we will check that
$$q_{(n,1,a-1,b)} - q^*_{(n,1,a,b)}+q_{(n-1,1,a,b)}\geq0, \eqno {(\dag)}  
$$
that is,
$$
   \frac  { {n+a-1\choose n}  {(b+1)} - r_{(n,1,a-1,b)}   }  {(n+2)}
-
\left \lceil   \frac{ {n+a\choose n}  {(b+1)}} { (n+2)  }\right  \rceil
$$
$$
+\frac  { {n+a-1\choose {n-1}}  {(b+1)} - r_{(n-1,1,a,b)}   }  {(n+1)}
\geq 0.
$$
For $n=3$ and $a=4$ we have
$q^*_{(3,1,4,b)} = q_{(3,1,4,b)} = 7(b+1)$,  and  $q_{(3,1,3,b)} = 4(b+1)$
hence
$$q_{(n,1,a-1,b)} - q^*_{(n,1,a,b)}+q_{(n-1,1,a,b)}=
4(b+1)-7(b+1)+\frac{ 15(b+1)-r _{(2,1,4,b)}}4
$$
$$ \geq \frac {1}4 (3(b+1)-3)\geq 0.
$$
Assume $(n,a) \neq (3,4)$. 
In order to prove that  $(\dag)$ holds,  it is enough to show that
$$   
\frac  { {n+a-1\choose n}  {(b+1)}  }  {(n+2)}
-
 \frac{ {n+a\choose n}  {(b+1)}} { (n+2)  }
+\frac  { {n+a-1\choose {n-1}}  {(b+1)}  }  {(n+1)}-2
\geq 0.
$$
This inequality is equivalent to the following one:
$${ n+a-1\choose {n-1}}  (b+1) -2 (n+1)(n+2)
\geq 0.
$$
For $n>3$ and $a=4$, since $b \geq1$,  we have
$${ n+a-1\choose {n-1}}  (b+1) -2 (n+1)(n+2)\geq 
2{ n+3\choose {4}} -2 (n+1)(n+2)
$$
$$= \frac{1}{24} (n+1)(n+2)((n+3)n-24) \geq  \frac{1}{24} (5)(6)(28-24) \geq 0.
$$
For  $n \geq 3$ and $a >4$, since $b \geq1$  we have
$${ n+a-1\choose {n-1}}  (b+1) -2 (n+1)(n+2)\geq 
{ n+4\choose {5}}  2 -2 (n+1)(n+2)
$$
$$= \frac { (n+1)(n+2)}{60} ( (n+4)(n+3)n- 120) \geq  \frac 1{3} (126- 120) \geq0
.
$$
So we have proved that  (3*) holds. 
\medskip

 Now we have to check  that (4) holds, that is:
 
  (4) $e^*_{(n,1,a-2,b)} \leq q_{(n,1,a,b)}-q_{(n-1,1,a,b)}-r_{(n-1,1,a,b)}.$
  
  We consider the following five cases:
  
   {Case (a)}:   $a>4$;
   
{Case (b)}:  $a=4$, $n>3$ and $b$ odd;

{Case (c)}:  $a=4$, $n=3$ and $b$ odd;

  {Case (d)}:  $a=4$, $n=3$ and $b$ even;
  
   {Case (e)}:  $a=4$, $n>3$ and $b$ even.

In Case   (a) by the induction hypothesis, in Cases  (b) and (c) by Theorem \ref{teoremaAbrescia1},   we get
$$e^*_{(n,1,a-2,b)} = q^*_{(n,1,a-2,b)}.
$$
Then in Cases (a), (b) and (c), in order to show that  (4) holds we have check that
 $$q^*_{(n,1,a-2,b)} - q_{(n,1,a,b)}+q_{(n-1,1,a,b)}+r_{(n-1,1,a,b)} \leq0, \eqno{(\ddag)} $$
that is,
$$
\left \lceil   \frac{ {n+a-2\choose n}  {(b+1)}} { (n+2)  }\right  \rceil
- 
   \frac  { {n+a\choose n}  {(b+1)} - r_{(n,1,a,b)}   }  {(n+2)}
   $$ $$
   +
   \frac  { {n-1+a\choose {n-1}}  {(b+1)} - r_{(n-1,1,a,b)}   }  {(n+1)}
   + r_{(n-1,1,a,b)} 
\leq 0.
$$

It is easy to show that $f(b,n,a) \geq 0 $ implies that the inequality $(\ddag)$ holds,
where
$$ f(b,n,a)=
(b+1){n-2+a\choose {n-1}} \left (n-  \frac {n -1} a \right) $$
$$-(n+1 )(n+2) -  r_{(n-1,1,a,b)}  n (n+2) .
$$
In Case (a), since  $b \geq1$, $a \geq5$, $ r_{(n-1,1,a,b)} \leq n$, and  $n \geq3$  we get
$$
 f(b,n,a) \geq 2{n+3\choose {4}}   \frac {4n+1} 5 -(n+1 )(n+2) -   n^2 (n+2) 
 $$
  $$
 =\frac {n+2} {60} ( n(n+1)(4n^2+13n+3 -60 )-60)
 $$
 $$
 \geq \frac {n+2} {60} (216-60) \geq0.
 $$
In Case (b)  (where $a=4$ and $n\geq4$), since  $b \geq1$, $ r_{(n-1,1,a,b)} \leq n$,  we get
$$ f(b,n,a)\geq
2{n+2\choose {3}}   \frac {3n +1} 4  -(n+1 )(n+2) -   n^2 (n+2) 
$$
$$=\frac {n+2} {12} ( n(n+1)(3n+1 -12 )-12)\geq 
\frac {6} {12} ( 20-12) \geq0
.$$
So ${(\ddag)} $ holds both in Case (a) and in Case (b).

In Case (c), since $a=4$ and $n= 3$, we have
$$q^*_{(n,1,a-2,b)}=q^*_{(3,1,2,b)}= 2(b+1); \ \  q_{(n,1,a,b)}= q_{(3,1,4,b)} = 7(b+1)
$$
$$q_{(n-1,1,a,b)}= q_{(2,1,4,b)} = \frac{15 (b+1)- r_{(2,1,4,b)}}  4.$$
Hence, since  $b \geq1$, $ r_{(2,1,4,b)} \leq 3$,  we get
 $$q^*_{(n,1,a-2,b)} -q_{(n,1,a,b)}+q_{(n-1,1,a,b)}+r_{(n-1,1,a,b)}$$
 $$=2(b+1)-7(b+1)+\frac{15 (b+1)- r_{(2,1,4,b)}}  4 +  r_{(2,1,4,b)}
 $$
  $$=\frac {1}4  (-5 (b+1)+3  r_{(2,1,4,b)}) \leq - \frac {1}4 .
 $$

 It follows that ${(\ddag)} $ holds in Case (c).
 
Now we are left with Cases (d) and (e). 
Since $e^*_{(n,1,a-2,b)} +q_{(n-1,1,a,b)}+r_{(n-1,1,a,b)}$ is an integer, recalling that $q_{(n,1,a,b)}=    \left  \lfloor  \frac  { {n+4\choose n}  {(2d+1)}   }  {(n+2)}  \right \rfloor$, then the inquality (4) is equivalent to the following:
  $$e^*_{(n,1,a-2,b)} +q_{(n-1,1,a,b)}+r_{(n-1,1,a,b)} \leq 
    \frac  { {n+4\choose n}  {(2d+1)}   }  {(n+2)} .$$
Let 
$g(n,4,2d)= e^*_{(n,1,a-2,b)} +q_{(n-1,1,a,b)}+r_{(n-1,1,a,b)} -
    \frac  { {n+4\choose n}  {(2d+1)}   }  {(n+2)}$. 
    
    Since Cases  (d) and (e), where  $b$ is even (say $b=2d$), and $a=4$, from Theorem \ref{teoremaAbrescia1} we have
$$e^*_{(n,1,a-2,b)} = (d+1)(n+1),$$
it follows that
$$g(n,4,2d)=(d+1)(n+1)
$$
$$+
 \frac  { {n-1+4\choose {4}}  {(2d+1)} - r_{(n-1,1,4,2d)}   }  {(n+1)} 
   + r_{(n-1,1,4,2d)} -
  \frac  { {n+4\choose 4}  {(2d+1)}   }  {(n+2)}  .
$$ 
 In Case (d) (where $n=3$), since $\frac {  15(2d+1)- r_{(2,1,4,2d)}}4$ is an integer and $0 \leq r_{(2,1,4,2d)} \leq 3$, if $d $ is even we have $ r_{(2,1,4,2d)} = 3$, if d is odd we have  $ r_{(2,1,4,2d)} = 1$. Hence 
 
$ \bullet$ for d even (so $d \geq2$) we have
 $$g(n,4,2d)  =  4(d+1)+ \frac {  15(2d+1)+9}4 - 7(2d+1) 
$$
 $$
=  \frac {12-10d}4 \leq -2 <0 .
$$
$ \bullet$ for d odd (so $d \geq1$) we have
$$g(n,4,2d)  =  4(d+1)+ \frac {  15(2d+1)+3}4- 7(2d+1)
$$
$$
=  \frac {6-10d}4 \leq - 1 <0 .
$$

 In Case (e), since $d \geq1$, $ r_{(n-1,1,4,2d) } \leq n$ and $n \geq4$, we get
 $$ g(n,4,2d)
= (d+1)(n+1)  
   +
   \frac  { {(n+3)(n+2)(n+1)n}  {(2d+1)}   }  {24(n+1)} 
    $$ 
   $$
   + \frac {n    r_{(n-1,1,4,2d)  }    } {n+1}
   - 
   \frac  { {(n+4)(n+3)(n+2)(n+1)}  {(2d+1)}  }  {24(n+2)}
$$
   $$= \frac {12 (n+1) -(2d+1)(3n^2+n) } {24} 
   + \frac {n    r_{(n-1,1,4,2d)  }    } {n+1}$$
    $$\leq \frac {12 (n+1) -3(3n^2+n) } {24} 
   + \frac {n ^2   } {n+1}$$
   $$=  \frac {9 n+12 -9n^2 } {24} 
   + \frac {n ^2-1   } {n+1}+ \frac {1   } {n+1}$$
$$=  \frac { -3n^2+11n +4} {8} 
   -1+ \frac {1   } {n+1}= - \frac {  (n-4)(3n+1)}8-   \frac {n } {n+1} <0 $$
Thus we have proved  that (4) holds in Cases (d) and (e). 

The conclusion now immediately follows from Lemma \ref{lemma3}.

\end{proof}


\section{Grassmann Defectivity}\label{Gd}

In this section we want to point out the consequences of Theorem  \ref{mainteor} in terms of Grassmann defectivity. First of all let us recall some definitions and some results.
 
\begin{defn}\label{GrassDef}   Let $X \subset \mathbb P^ N$ be a closed,  irreducible 
and non-degenerate projective variety of dimension $n$.  Let $0 \leq k \leq s-1 < N$ be integers and let $\mathbb{G}(k,N)$ be the Grassmannian of the $\mathbb{P}^k$'s contained in $\mathbb{P}^N$.

The {\it Grassmann secant variety}  denoted with $Sec_{k,s-1}(X)$  is the Zariski closure of the set
$$\{\Lambda \in \mathbb{G}(k, N) |   \Lambda \hbox{ lies in the linear span of }  s \ \hbox  {independent points of } X \}.$$
\end{defn}

The expected dimension of this variety is 
$$exp \dim(Sec_{k,s-1}(X))=\min\{sn + (k + 1)(s-1 - k), (k + 1)(N - k)\}.$$
Analogously to the classical secant varieties  we define the $(k,s-1)$-defect of $X$ as the number: 
$$
\delta_{k,s-1}(X)= exp \dim(Sec_{k,s-1}(X))-\dim
Sec_{k,s-1}(X). 
$$
(For general information about these defectivities see\cite{ChCo} and  \cite{DF03}).

In \cite{DF03} C. Dionisi e C. Fontanari prove the following proposition, previously proved by Terracini in \cite{Ter15}  in the specific case in which $X$ is a Veronese surface.

\begin{prop}\cite [Proposition 1.3]{DF03} \label{CarlaClaudio}
Let $X\subset  \mathbb{P}^N$ be an irreducible non-degenerate projective
variety of dimension $n$. Let $\phi :\mathbb{P}^k \times X \rightarrow \mathbb{P}^{N(k+1)+k}$ be the Segre
embedding of  $\mathbb{P}^k \times X$. Then $X$ is $(k, s-1)$-defective with defect $\delta_{k,s-1}(X)=\delta$ if and only if $\phi (\mathbb{P}^k \times X)$ is $(s-1)$-defective with defect $\delta_{s-1}(\phi(\mathbb{P}^k \times X))=\delta$.
\end{prop}

Hence the following corollary arises as natural consequence of Theorem \ref{mainteor} and Proposition \ref{CarlaClaudio}.

\begin{cor}\label{cor1} The Veronese varieties $\nu_a(\mathbb{P}^n)$ are $(1,s-1)$-defective if and only if $n=2$, $a=3$ and $s=5$. 
\end{cor}

\begin{proof} Apply Theorem \ref{mainteor},  in the specific case of the Segre-Veronese embedding of $\mathbb{P}^n\times \mathbb{P}^1$ with $\mathcal{O}(a,1)$. Then use  Proposition \ref{CarlaClaudio}.
\end{proof}


Following Terracini's paper, we can  translate the previous result in terms of the number of homogeneous polynomials of certain degree $a$ that can be written as linear combination of  $a$-th powers of the same linear forms; in fact, the problem of finding when $Sec_{k, s-1}(\nu_a(\mathbb{P}^n))$ coincides with the Grassmannian $\mathbb{G}(k,N)$, ($N={n+a\choose a}-1$), is equivalent to a question related to the decomposition of $k+1$ general forms of degree $a$ as linear combination of $a$-th powers of the same $s$ linear forms. The classical version of this problem is originally due to A. Terracini (see \cite {Ter15}) and J. Bronowski (see \cite {Br}). In terms of forms the Corollary \ref{cor1} can be translated as follows:
 
\begin{rem} \label{Ter}
The closure of the set of all pairs of homogeneous degree $a$ polynomials
in $n+1$ variables which may be written as linear combinations of 
$a$-th powers of the same $s$ linear forms
 $L_1, \ldots ,L_s \in K[ x_0, \ldots , x_n]_1$ has the expect dimension if and only if  $(n, a,s)\neq (2,3,5)$.
 \end{rem}

\textbf{Acknowledgments} All authors supported by MIUR funds. The second author was partially supported by CIRM--FBK (Centro Internazionale Ricerca Matematica--Fondazione Bruno Kessler), Trento, Italy; by Project Galaad, INRIA (Institut national de recherche en informatique et en automatique), Sophia Antipolis M\'editerran\'ee, France; Institut Mittag-Leffler (Sweden); and by International Individual Marie Curie Fellowship (IFP7 Program). 




\end{document}